\documentclass[11pt,twoside]{amsart}
\usepackage{amsmath, amsthm, amscd, amsfonts, amssymb, graphicx}
\usepackage{enumerate}
\usepackage[colorlinks=true,
linkcolor=blue,
urlcolor=cyan,
citecolor=red]{hyperref}
\usepackage{mathrsfs}
\newtheorem{theorem}{Theorem}[section]
\newtheorem{lemma}{Lemma}[section]

\newtheorem{corollary}{Corollary}[section]

\numberwithin{equation}{section}

\begin{document}
	\title{New inequalities for operator concave functions involving positive linear maps}
	\author{Shiva Sheybani, Mohsen Erfanian Omidvar, and Hamid Reza Moradi}
	 \maketitle
	\begin{abstract}
		The purpose of this paper is to present some general inequalities for operator concave functions which include some known inequalities as a particular case. Among other things, we prove that if $A\in \mathcal{B}\left( \mathcal{H} \right)$ is a positive operator such that $mI\le A\le MI$ for some scalars $0<m<M$ and $\Phi $ is a normalized positive linear map on $\mathcal{B}\left( \mathcal{H} \right)$, then
		\[\begin{aligned}
		{{\left( \frac{M+m}{2\sqrt{Mm}} \right)}^{r}}&\ge {{\left( \frac{\frac{1}{\sqrt{Mm}}\Phi \left( A \right)+\sqrt{Mm}\Phi \left( {{A}^{-1}} \right)}{2} \right)}^{r}} \\ 
		& \ge \frac{\frac{1}{{{\left( Mm \right)}^{\frac{r}{2}}}}\Phi {{\left( A \right)}^{r}}+{{\left( Mm \right)}^{\frac{r}{2}}}\Phi {{\left( {{A}^{-1}} \right)}^{r}}}{2} \\ 
		& \ge \Phi {{\left( A \right)}^{r}}\sharp\Phi {{\left( {{A}^{-1}} \right)}^{r}},  
		\end{aligned}\]
		where $0\le r\le 1$, which nicely extend the operator Kantorovich inequality.
	\end{abstract}
	\pagestyle{myheadings}
	\markboth{\centerline {New inequalities for operator concave functions involving positive linear maps}}
	{\centerline {S. Sheybani, M.E. Omidvar \& H.R. Moradi}}
	\bigskip
	\bigskip
	\section{\bf Introduction and preliminaries}
	\vskip 0.4 true cm
	In this paper we consider operator monotone and convex functions defined on the half
	real line $\left( 0,\infty  \right)$. Let $\mathcal{B}\left( \mathcal{H} \right)$ be the algebra of all bounded linear operators on a complex Hilbert space   and $I$ denote the identity operator. If $A$ is an operator then we denote $Sp\left( A \right)$ its spectrum. An operator $A$ is called {\it positive} if $\left\langle Ax,x \right\rangle \ge 0$ for all $x\in \mathcal{H}$, and we then write $A\ge 0$. By $B\ge A$ we mean that $B-A$ is positive, while $B>A$ means that $B-A$ is strictly positive. A mapping $\Phi $ on $\mathcal{B}\left( \mathcal{H} \right)$ is said to be {\it positive} if $\Phi \left( A \right)\ge 0$ for each $A\ge 0$ and is called {\it normalized} if $\Phi $ preserves the identity operator.
	
	For any strictly positive operator $A,B\in \mathcal{B}\left( \mathcal{H} \right)$ and $v\in \left[ 0,1 \right]$, we write
	\[A{{\nabla }_{v}}B:=\left( 1-v \right)A+vB\quad\text{ and }\quad A{{\sharp}_{v}}B:={{A}^{\frac{1}{2}}}{{\left( {{A}^{-\frac{1}{2}}}B{{A}^{-\frac{1}{2}}} \right)}^{v}}{{A}^{\frac{1}{2}}}.\] 
	For the case $v=\frac{1}{2}$, we write $\nabla $ and $\sharp$, respectively. The operator arithmetic-geometric mean inequality (in short, AM-GM inequality) asserts that $A{{\sharp}_{v}}B\le A{{\nabla }_{v}}B$, for any positive operators $A,B\in \mathcal{B}\left( \mathcal{H} \right)$ and any $v\in \left[ 0,1 \right]$. A real valued function $f$ defined on an interval $J$ is said to be {\it operator convex} (resp. {\it operator concave}) if $f\left( A{{\nabla }_{v}}B \right)\le f\left( A \right){{\nabla }_{v}}f\left( B \right)$ (resp. $f\left( A{{\nabla }_{v}}B \right)\ge f\left( A \right){{\nabla }_{v}}f\left( B \right)$) for all self-adjoint operators $A,B$ with spectra in $J$ and all $v\in \left[ 0,1 \right]$. A continuous real valued function $f$ defined on an interval $J$ is called {\it operator monotone} (more precisely, {\it operator monotone increasing}) if $B\ge A$
	implies that $f\left( B \right)\ge f\left( A \right)$, and {\it operator monotone decreasing} if  $B\ge A$ implies $f\left( B \right)\le f\left( A \right)$ for all self-adjoint operators $A,B$ with spectra in $J$.
	
	During the past decades several formulations, extensions or refinements of the Kantorovich inequality \cite{02} in various settings have been introduced by many mathematicians; see \cite{lin,6,moradi,10} and references therein.
	
	Let $A\in \mathcal{B}\left( \mathcal{H} \right)$ be a positive operator such that $mI\le A\le MI$ for some scalars $0<m<M$ and $\Phi $ be a normalized positive linear map on $\mathcal{B}\left( \mathcal{H} \right)$, then
	\begin{equation}\label{16}
	\Phi \left( {{A}^{-1}} \right)\sharp \Phi \left( A \right)\le \frac{M+m}{2\sqrt{Mm}}.
	\end{equation}
	In addition
	\begin{equation}\label{15}
	\Phi \left( A \right)\sharp \Phi \left( B \right)\le \frac{M+m}{2\sqrt{Mm}}\Phi \left( A\sharp B \right),
	\end{equation}
	whenever ${{m}^{2}}A\le B\le {{M}^{2}}A$ and $0<m<M$. The first inequality goes back to Nakamoto and Nakamura in the 1996's \cite{9}, the second is more general and has been proved only in 2009 by Lee \cite{8} (its matrix version).
	
	In Sec. \ref{s2}, we first extend \eqref{15}, then as an application, we obtain a generalization of \eqref{16}. In Sec. \ref{s3}, we use elementary operations and give some inequalities related to the Bellman type.
	\section{\bf Some operator inequalities involving positive linear map}\label{s2}
	\vskip 0.4 true cm
	We prove the following new result, from which \eqref{15} directly follows:
	\begin{theorem}\label{13}
		Let $A,B\in \mathcal{B}\left( \mathcal{H} \right)$ be two strictly positive operators such that $m_{1}^{2}I\le A\le M_{1}^{2}I$, $m_{2}^{2}I\le B\le M_{2}^{2}I$ for some positive scalars ${{m}_{1}}<{{M}_{1}}$, ${{m}_{2}}<{{M}_{2}}$, and let $\Phi $ be a normalized positive linear map on $\mathcal{B}\left( \mathcal{H} \right)$. If $f$ is an operator monotone, then   
		\[\begin{aligned}
		f\left( \left( \frac{M+m}{2} \right)\Phi \left( A\sharp B \right) \right)&\ge f\left( \frac{Mm\Phi \left( A \right)+\Phi \left( B \right)}{2} \right) \\ 
		& \ge \frac{f\left( Mm\Phi \left( A \right) \right)+f\left( \Phi \left( B \right) \right)}{2} \\ 
		& \ge f\left( Mm\Phi \left( A \right) \right)\sharp f\left( \Phi \left( B \right) \right), 
		\end{aligned}\]
		where $m=\frac{{{m}_{2}}}{{{M}_{1}}}$ and $M=\frac{{{M}_{2}}}{{{m}_{1}}}$.
	\end{theorem}
	\begin{proof}
		According to the assumption, we have
		\begin{equation*}
		mI\le {{\left( {{A}^{-\frac{1}{2}}}B{{A}^{-\frac{1}{2}}} \right)}^{\frac{1}{2}}}\le MI,
		\end{equation*}
		it follows that
		\begin{equation*}
		\left( M+m \right){{\left( {{A}^{-\frac{1}{2}}}B{{A}^{-\frac{1}{2}}} \right)}^{\frac{1}{2}}}\ge MmI+{{A}^{-\frac{1}{2}}}B{{A}^{-\frac{1}{2}}}.
		\end{equation*}
		The above inequality then implies
		\[\left( \frac{M+m}{2} \right)A\sharp B\ge \frac{MmA+B}{2}.\]
		Using  the  hypotheses  made  about $\Phi $, 
		\[\left( \frac{M+m}{2} \right)\Phi \left( A\sharp B \right)\ge \frac{Mm\Phi \left( A \right)+\Phi \left( B \right)}{2}.\]
		Thus we have
		\[\begin{aligned}
		f\left( \left( \frac{M+m}{2} \right)\Phi \left( A\sharp B \right) \right)&\ge f\left( \frac{Mm\Phi \left( A \right)+\Phi \left( B \right)}{2} \right) \quad \text{(since $f$ is operator monotone)}\\ 
		& \ge \frac{f\left( Mm\Phi \left( A \right) \right)+f\left( \Phi \left( B \right) \right)}{2} \quad \text{(by \cite[Corollary 1.12]{1})}\\ 
		& \ge f\left( Mm\Phi \left( A \right) \right)\sharp f\left( \Phi \left( B \right) \right)  \quad \text{(by AM-GM inequality)},
		\end{aligned}\]
		which is the statement of the theorem.
	\end{proof}
	We complement Theorem \ref{13} by proving the following.
	\begin{theorem}\label{14}
		Let $A,B\in \mathcal{B}\left( \mathcal{H} \right)$ be two strictly positive operators such that $m_{1}^{2}I\le A\le M_{1}^{2}$I, $m_{2}^{2}I\le B\le M_{2}^{2}I$ for some scalars ${{m}_{1}}<{{M}_{1}}$, ${{m}_{2}}<{{M}_{2}}$, and let $\Phi $ be a normalized positive linear map on $\mathcal{B}\left( \mathcal{H} \right)$. If $g$ is an operator monotone decreasing, then  
		\[\begin{aligned}
		g\left( \left( \frac{M+m}{2} \right)\Phi \left( A\sharp B \right) \right)&\le g\left( \frac{Mm\Phi \left( A \right)+\Phi \left( B \right)}{2} \right) \\ 
		& \le {{\left\{ \frac{g{{\left( Mm\Phi \left( A \right) \right)}^{-1}}+g{{\left( \Phi \left( B \right) \right)}^{-1}}}{2} \right\}}^{-1}} \\ 
		& \le g\left( Mm\Phi \left( A \right) \right)\sharp g\left( \Phi \left( B \right) \right), 
		\end{aligned}\]
		where $m=\frac{{{m}_{2}}}{{{M}_{1}}}$ and $M=\frac{{{M}_{2}}}{{{m}_{1}}}$.
	\end{theorem}
	\begin{proof}
		Since $g$ is operator monotone decreasing on $\left( 0,\infty  \right)$, so $\frac{1}{g}$ is operator monotone on $\left( 0,\infty  \right)$. Now by applying Theorem \ref{13} for $f=\frac{1}{g}$, we have
		\[\begin{aligned}
		g{{\left( \left( \frac{M+m}{2} \right)\Phi \left( A\sharp B \right) \right)}^{-1}}&\ge g{{\left( \frac{Mm\Phi \left( A \right)+\Phi \left( B \right)}{2} \right)}^{-1}} \\ 
		& \ge \frac{g{{\left( Mm\Phi \left( A \right) \right)}^{-1}}+g{{\left( \Phi \left( B \right) \right)}^{-1}}}{2} \\ 
		& \ge g{{\left( Mm\Phi \left( A \right) \right)}^{-1}}\sharp g{{\left( \Phi \left( B \right) \right)}^{-1}}.  
		\end{aligned}\]
		Taking the inverse, we get
		\[\begin{aligned}
		g\left( \left( \frac{M+m}{2} \right)\Phi \left( A\sharp B \right) \right)&\le g\left( \frac{Mm\Phi \left( A \right)+\Phi \left( B \right)}{2} \right) \\ 
		& \le {{\left\{ \frac{g{{\left( Mm\Phi \left( A \right) \right)}^{-1}}+g{{\left( \Phi \left( B \right) \right)}^{-1}}}{2} \right\}}^{-1}} \\ 
		& \le {{\left\{ g{{\left( Mm\Phi \left( A \right) \right)}^{-1}}\sharp g{{\left( \Phi \left( B \right) \right)}^{-1}} \right\}}^{-1}} \\ 
		& =g\left( Mm\Phi \left( A \right) \right)\sharp g\left( \Phi \left( B \right) \right), 
		\end{aligned}\]
		proving the main assertion of the theorem.
	\end{proof}
	As a byproduct of Theorems \ref{13} and \ref{14}, we have the following result.
	\begin{corollary}\label{19}
		Under the assumptions of Theorem \ref{13}.
		\begin{itemize}
			\item[(i)]  If $0\le r\le 1$, then
			\[\begin{aligned}
			{{\left( \frac{M+m}{2\sqrt{Mm}} \right)}^{r}}\Phi {{\left( A\sharp B \right)}^{r}}&\ge {{\left( \frac{Mm\Phi \left( A \right)+\Phi \left( B \right)}{2\sqrt{Mm}} \right)}^{r}} \\ 
			& \ge \frac{{{\left( Mm \right)}^{r}}\Phi {{\left( A \right)}^{r}}+\Phi {{\left( B \right)}^{r}}}{2{{\left( Mm \right)}^{\frac{r}{2}}}} \\ 
			& \ge \Phi {{\left( A \right)}^{r}}\sharp \Phi {{\left( B \right)}^{r}}.  
			\end{aligned}\]
			The important special case
			\[\frac{M+m}{2\sqrt{Mm}}\Phi \left( A\sharp B \right)\ge \frac{Mm\Phi \left( A \right)+\Phi \left( B \right)}{2\sqrt{Mm}}\ge \Phi \left( A \right)\sharp\Phi \left( B \right),\]
			was observed by Moslehian et al. \cite{10} (see \cite[Theorem 2.5]{moradi} for much stronger result). 
			\item[(ii)] If $-1\le r\le 0$, then
			\[\begin{aligned}
			{{\left( \frac{M+m}{2\sqrt{Mm}} \right)}^{r}}\Phi {{\left( A\sharp B \right)}^{r}}&\le {{\left( \frac{Mm\Phi \left( A \right)+\Phi \left( B \right)}{2\sqrt{Mm}} \right)}^{r}} \\ 
			& \le \frac{1}{{{\left( Mm \right)}^{\frac{r}{2}}}}{{\left\{ \frac{{{\left( Mm \right)}^{-r}}\Phi {{\left( A \right)}^{-r}}+\Phi {{\left( B \right)}^{-r}}}{2} \right\}}^{-1}} \\ 
			& \le \Phi {{\left( A \right)}^{r}}\sharp\Phi {{\left( B \right)}^{r}}.  
			\end{aligned}\]
			
		\end{itemize}
	\end{corollary}
	Our next result is a straightforward application of Theorems \ref{13} and \ref{14}.
	\begin{corollary}\label{02}
		Let $A\in \mathcal{B}\left( \mathcal{H} \right)$ be positive operator such that $mI\le A\le MI$ for some scalars $0<m<M$ and $\Phi $ be a normalized positive linear map on $\mathcal{B}\left( \mathcal{H} \right)$.
		\begin{itemize}
			\item[(i)]  If $f$ is an operator monotone, then
			\[\begin{aligned}
			f\left( \frac{M+m}{2Mm} \right)&\ge f\left( \frac{\frac{1}{Mm}\Phi \left( A \right)+\Phi \left( {{A}^{-1}} \right)}{2} \right) \\ 
			& \ge \frac{f\left( \frac{1}{Mm}\Phi \left( A \right) \right)+f\left( \Phi \left( {{A}^{-1}} \right) \right)}{2} \\ 
			& \ge f\left( \frac{1}{Mm}\Phi \left( A \right) \right)\sharp f\left( \Phi \left( {{A}^{-1}} \right) \right).  
			\end{aligned}\]
			\item[(ii)] If $g$ is an operator monotone decreasing, then
			\[\begin{aligned}
			g\left( \frac{M+m}{2Mm} \right)&\le g\left( \frac{\frac{1}{Mm}\Phi \left( A \right)+\Phi \left( {{A}^{-1}} \right)}{2} \right) \\ 
			& \le {{\left\{ \frac{g{{\left( \frac{1}{Mm}\Phi \left( A \right) \right)}^{-1}}+g{{\left( \Phi \left( {{A}^{-1}} \right) \right)}^{-1}}}{2} \right\}}^{-1}} \\ 
			& \le g\left( \frac{1}{Mm}\Phi \left( A \right) \right)\sharp g\left( \Phi \left( {{A}^{-1}} \right) \right).  
			\end{aligned}\]
		\end{itemize}
	\end{corollary}
	In the same vein as in Corollary \ref{19}, we have the following consequences.
	\begin{corollary}
		Under the assumptions of Corollary \ref{02}.
		\begin{itemize}
			
			\item[(i)] If $0\le r\le 1$, then
			\[\begin{aligned}
			{{\left( \frac{M+m}{2\sqrt{Mm}} \right)}^{r}}&\ge {{\left( \frac{\frac{1}{\sqrt{Mm}}\Phi \left( A \right)+\sqrt{Mm}\Phi \left( {{A}^{-1}} \right)}{2} \right)}^{r}} \\ 
			& \ge \frac{\frac{1}{{{\left( Mm \right)}^{\frac{r}{2}}}}\Phi {{\left( A \right)}^{r}}+{{\left( Mm \right)}^{\frac{r}{2}}}\Phi {{\left( {{A}^{-1}} \right)}^{r}}}{2} \\ 
			& \ge \Phi {{\left( A \right)}^{r}}\sharp\Phi {{\left( {{A}^{-1}} \right)}^{r}}.  
			\end{aligned}\]
			For the special case in which $r=1$, we have 
			\[\frac{M+m}{2\sqrt{Mm}}\ge \frac{\frac{1}{\sqrt{Mm}}\Phi \left( A \right)+\sqrt{Mm}\Phi \left( {{A}^{-1}} \right)}{2}\ge \Phi \left( A \right)\sharp\Phi \left( {{A}^{-1}} \right).\]	
			\item[(ii)] If $-1\le r\le 0$, then
			\[\begin{aligned}
			{{\left( \frac{M+m}{2\sqrt{Mm}} \right)}^{r}}&\le {{\left( \frac{\frac{1}{\sqrt{Mm}}\Phi \left( A \right)+\sqrt{Mm}\Phi \left( {{A}^{-1}} \right)}{2} \right)}^{r}} \\ 
			& \le {{\left\{ \frac{{{\left( Mm \right)}^{r}}\Phi {{\left( A \right)}^{-r}}+\Phi {{\left( {{A}^{-1}} \right)}^{-r}}}{2{{\left( Mm \right)}^{\frac{r}{2}}}} \right\}}^{-1}} \\ 
			& \le \Phi {{\left( A \right)}^{r}}\sharp\Phi {{\left( {{A}^{-1}} \right)}^{r}}.  
			\end{aligned}\]
		\end{itemize}
	\end{corollary}
	\section{\bf Operator Bellman inequality with negative parameter}\label{s3}
	\vskip 0.4 true cm
	Let $A,B\in \mathcal{B}\left( \mathcal{H} \right)$ be two strictly positive operators and $\Phi $ be a normalized positive linear map on $\mathcal{B}\left( \mathcal{H} \right)$. If $f$ is an operator concave, then for any $v\in \left[ 0,1 \right]$, the following inequality obtained in \cite[Theorem 2.1]{2}:
	\begin{equation}\label{7}
	\Phi \left( f\left( A \right) \right){{\nabla }_{v}}\Phi \left( f\left( B \right) \right)\le f\left( \Phi \left( A{{\nabla }_{v}}B \right) \right).
	\end{equation}
	In the same paper, as an operator version of Bellman inequality \cite{bellman}, the authors showed that
	\begin{equation}\label{17}
	\Phi \left( {{\left( I-A \right)}^{r}}{{\nabla }_{v}}{{\left( I-B \right)}^{r}} \right)\le \Phi {{\left( I-A{{\nabla }_{v}}B \right)}^{r}},
	\end{equation}
	where  $A,B$ are two operator contractions (in the sense that $\left\| A \right\|,\left\| B \right\|\le 1$) and $r,v\in \left[ 0,1 \right]$. 
	
	Under the convexity assumption on $f$, \eqref{17} can be reversed:
	\begin{theorem}
		Let $A,B\in \mathcal{B}\left( \mathcal{H} \right)$ be two contraction operators and $\Phi $ be a normalized positive linear map on $\mathcal{B}\left( \mathcal{H} \right)$. Then
		\begin{equation}\label{2}
		\Phi {{\left( I-A{{\nabla }_{v}}B \right)}^{r}}\le \Phi \left( {{\left( I-A \right)}^{r}}{{\nabla }_{v}}{{\left( I-B \right)}^{r}} \right),
		\end{equation}
		for any $v\in \left[ 0,1 \right]$ and $r\in \left[ -1,0 \right]\cup \left[ 1,2 \right]$.
	\end{theorem}
	\begin{proof}
		If $f$ is operator convex, we have
		\[\begin{aligned}
		f\left( \Phi \left( A{{\nabla }_{v}}B \right) \right)&\le \Phi \left( f\left( A{{\nabla }_{v}}B \right) \right) \quad \text{(by Choi-Davis-Jensen inequality \cite[p. 62]{4})}\\ 
		& \le \Phi \left( f\left( A \right){{\nabla }_{v}}f\left( B \right) \right) \quad \text{(by operator convexity of $f$)}. 
		\end{aligned}\]
		The function $f\left( t \right)={{t}^{r}}$ is operator convex on $\left( 0,\infty  \right)$ for $r\in \left[ -1,0 \right]\cup \left[ 1,2 \right]$ (see \cite[Chapter 1]{4}). It can be verified that $f\left( t \right)={{\left( 1-t \right)}^{r}}$ is operator convex on $\left( 0,1  \right)$ for $r\in \left[ -1,0 \right]\cup \left[ 1,2 \right]$. This implies the desired result \eqref{2}.
	\end{proof}
	However, we are looking for something stronger than \eqref{2}. The principal object of this section is to prove the following:
	\begin{theorem}\label{18}
		Let $A,B\in \mathcal{B}\left( \mathcal{H} \right)$ be two contraction operators and $\Phi $ be a normalized positive linear map on $\mathcal{B}\left( \mathcal{H} \right)$. Then
		\[\begin{aligned}
		\Phi {{\left( I-A{{\nabla }_{v}}B \right)}^{r}}&\le \Phi {{\left( I-A \right)}^{r}}{{\sharp}_{v}}\Phi {{\left( I-B \right)}^{r}} \\ 
		& \le \Phi \left( {{\left( I-A \right)}^{r}}{{\sharp}_{v}}{{\left( I-B \right)}^{r}} \right) \\ 
		& \le \Phi \left( {{\left( I-A \right)}^{r}}{{\nabla }_{v}}{{\left( I-B \right)}^{r}} \right), 
		\end{aligned}\]
		where $v\in \left[ 0,1 \right]$ and $r\in \left[ -1,0 \right]$.
	\end{theorem}
	The proof is at the end of this section. The following lemma will play an important role in our proof. 
	\begin{lemma}\label{3}
		Let $A,B\in \mathcal{B}\left( \mathcal{H} \right)$ be two strictly positive operators and $\Phi $ be a normalized positive linear map on $\mathcal{B}\left( \mathcal{H} \right)$. If $f$ is an operator monotone decreasing, then for any $v\in \left[ 0,1 \right]$
		\begin{equation}\label{5}
		f\left( \Phi \left( A{{\nabla }_{v}}B \right) \right)\le f\left( \Phi \left( A \right) \right){{\sharp}_{v}}f\left( \Phi \left( B \right) \right)\le \Phi \left( f\left( A \right) \right){{\nabla }_{v}}\Phi \left( f\left( B \right) \right),
		\end{equation}
		and
		\begin{equation}\label{6}
		f\left( \Phi \left( A{{\nabla }_{v}}B \right) \right)\le \Phi \left( f\left( A \right){{\sharp}_{v}}f\left( B \right) \right)\le \Phi \left( f\left( A \right) \right){{\nabla }_{v}}\Phi \left( f\left( B \right) \right).
		\end{equation}
		More precisely, 
		\begin{equation}\label{20}
		f\left( \Phi \left( A{{\nabla }_{v}}B \right) \right)\le f\left( \Phi \left( A \right) \right){{\sharp}_{v}}f\left( \Phi \left( B \right) \right)\le \Phi \left( f\left( A \right){{\sharp}_{v}}f\left( B \right) \right)\le \Phi \left( f\left( A \right) \right){{\nabla }_{v}}\Phi \left( f\left( B \right) \right).	
		\end{equation}
	\end{lemma}

	\begin{proof}
		As Ando and Hiai mentioned in \cite[(2.16)]{1}, the function $f$ is an operator monotone decreasing if and only if 
		\begin{equation}\label{4}
		f\left( A{{\nabla }_{v}}B \right)\le f\left( A \right){{\sharp}_{v}}f\left( B \right).
		\end{equation}
		We emphasize here that if $f$ satisfies in \eqref{4}, then is operator convex (this class of functions is called {\it operator log-convex}). It is easily verified that if $Sp\left( A \right),Sp\left( B \right)\subseteq J$, then $Sp\left( \Phi \left( A \right) \right),Sp\left( \Phi \left( B \right) \right)\subseteq J$. So we can replace $A,B$ by $\Phi \left( A \right),\Phi \left( B \right)$ in \eqref{4}, respectively. Therefore we can write
		\[\begin{aligned}
		& f\left( \Phi \left( A{{\nabla }_{v}}B \right) \right)\\ &\le f\left( \Phi \left( A \right) \right){{\sharp}_{v}}f\left( \Phi \left( B \right) \right) \\ 
		& \le \Phi \left( f\left( A \right) \right){{\sharp}_{v}}\Phi \left( f\left( B \right) \right) \quad \text{(by  Choi-Davis-Jensen inequality and monotonicity property of mean)} \\ 
		& \le \Phi \left( f\left( A \right){{\nabla }_{v}}f\left( B \right) \right)  \quad \text{(by AM-GM inequality)}.
		\end{aligned}\]
		This completes the proof of the inequality \eqref{5}. To prove the inequality \eqref{6}, note that if $Sp\left( A \right),Sp\left( B \right)\subseteq J$, then $Sp\left( A{{\nabla }_{v}}B \right)\subseteq J$. By computation
		\[\begin{aligned}
		f\left( \Phi \left( A{{\nabla }_{v}}B \right) \right)&\le \Phi \left( f\left( A{{\nabla }_{v}}B \right) \right) \quad \text{(by Choi-Davis-Jensen inequality)}\\ 
		& \le \Phi \left( f\left( A \right){{\sharp}_{v}}f\left( B \right) \right) \quad \text{(by \eqref{4})}\\ 
		& \le \Phi \left( f\left( A \right) \right){{\sharp}_{v}}\Phi \left( f\left( B \right) \right) \quad \text{(by Ando's inequality \cite[Theorem 3]{3})}\\ 
		& \le \Phi \left( f\left( A \right){{\nabla }_{v}}f\left( B \right) \right)  \quad \text{(by AM-GM inequality)},
		\end{aligned}\]
		proving the inequality \eqref{6}. We know that if $g$ is operator monotone on $\left( 0,\infty  \right)$, then $g$  is operator concave. As before, it can be shown that
		\[g\left( \Phi \left( A \right) \right){{\sharp}_{v}}g\left( \Phi \left( B \right) \right)\ge \Phi \left( g\left( A \right) \right){{\sharp}_{v}}\Phi \left( g\left( B \right) \right)\ge \Phi \left( g\left( A \right){{\sharp}_{v}}g\left( B \right) \right).\]
		Taking the inverse, we get
		\[g{{\left( \Phi \left( A \right) \right)}^{-1}}{{\sharp}_{v}}g{{\left( \Phi \left( B \right) \right)}^{-1}}\le \Phi {{\left( g\left( A \right){{\sharp}_{v}}g\left( B \right) \right)}^{-1}}\le \Phi \left( g{{\left( A \right)}^{-1}}{{\sharp}_{v}}g{{\left( B \right)}^{-1}} \right).\]
		If $g$ is operator monotone, then $f=\frac{1}{g}$ is operator monotone decreasing, we conclude
		\[f\left( \Phi \left( A \right) \right){{\sharp}_{v}}f\left( \Phi \left( B \right) \right)\le \Phi \left( f\left( A \right){{\sharp}_{v}}f\left( B \right) \right).\] 
		This proves \eqref{20}. 
	\end{proof}
	We are now in a position to present a proof of Theorem \ref{18}.
	\vskip 0.4 true cm
	\noindent {\it Proof of Theorem 3.2.}
	It is well-known that the function $f\left( t \right)={{t}^{r}}$ on $\left( 0,\infty  \right)$ is operator monotone decreasing for $r\in \left[ -1,0 \right]$. It implies that the function $f\left( t \right)={{\left( 1-t \right)}^{r}}$ on $\left( 0,1 \right)$ is operator monotone decreasing too. By applying Lemma \ref{3}, we get the desired result.  $\hfill\square$
	\vskip 0.5 true cm
	\noindent{\bf Acknowledgements.}
	The authors are thankful to the referee(s) for the useful comments and suggestions. The authors would also like to thank Professor Shigeru Furuichi for fruitful discussions.
	\bibliographystyle{alpha}
	
	\vskip 0.4 true cm
	
	\tiny(S. Sheybani) Department of Mathematics, Mashhad Branch, Islamic Azad University, Mashhad, Iran.
	
	{\it E-mail address:} shiva.sheybani95@gmail.com
	
		\vskip 0.4 true cm	
	
	(M.E. Omidvar) Department of Mathematics, Mashhad Branch, Islamic Azad University, Mashhad, Iran.
	
	{\it E-mail address:} erfanian@mshdiau.ac.ir
	
	\vskip 0.4 true cm
	
	\tiny(H.R. Moradi) Young Researchers and Elite Club, Mashhad Branch, Islamic Azad University, Mashhad, Iran.
	
	{\it E-mail address:} hrmoradi@mshdiau.ac.ir
	
\end{document}